\documentclass[11pt]{article}
\usepackage{amssymb,amsmath,amsthm}
\usepackage{epsfig}
\usepackage{psfrag}
\newtheorem{Theorem}{\sc Theorem}[section]

\newtheorem{lemma}[Theorem]{\sc Lemma}
\newtheorem{Proposition}[Theorem]{\sc Proposition}
\newtheorem{Corollary}[Theorem]{\sc Corollary}

 \newtheorem{Example}[Theorem]{\sc Example}
\newtheorem{Remark}[Theorem]{\sc Remark}

\def\hpic #1 #2 {\mbox{$\begin{array}[c]{l} 
\epsfig{file=#1,height=#2}\end{array}$}}
\def\wpic #1 #2 {\mbox{$\begin{array}[c]{l} 
\epsfig{file=#1,width=#2}\end{array}$}}

\def\C{\mathbb C}

\def\F{\mathbb F}

\def\N{\mathbb N}

\def\R{\mathbb R}

\def\Z{\mathbb Z}

\def\CF{{\cal {F}}}

\def\CH{{\cal {H}}}

\def\CL{{\cal {L}}}
\def\CM{{\cal {M}}}

\def\CP{{\cal {P}}}

\def\be{\begin{equation}}
\def\ee{\end{equation}}
\def\bt{\begin{Theorem}}
\def\et{\end{Theorem}}
\def\ba{\begin{array}}
\def\ea{\end{array}}
\def\bi{\begin{itemize}}
\def\ei{\end{itemize}}
\def\bea{\begin{eqnarray}}
\def\eea{\end{eqnarray}}
\def\beast{\begin{eqnarray*}}
\def\eeast{\end{eqnarray*}}
\def\ben{\begin{enumerate}}
\def\een{\end{enumerate}}

\def\half{\frac{1}{2}}

\def\bi{\bibitem}

\def\rar{\rightarrow}
\def\Rar{\Rightarrow}

\title{From graphs to free products}

\author{Madhushree Basu, Vijay Kodiyalam and V.S.Sunder\\The Institute of Mathematical
  Sciences, Chennai, India\\{\small e-mail:
    madhushree@imsc.res.in, vijay@imsc.res.in, sunder@imsc.res.in}}

\begin{document}

\maketitle


\begin{abstract}
We investigate a construction which associates a finite von Neumann
algebra $M(\Gamma,\mu)$ to a finite weighted graph
$(\Gamma,\mu)$. Pleasantly, but not surprisingly, the von Neumann
algebra associated to 
to a `flower with $n$ petals' is the group von Neumann algebra of the
free group on $n$ generators. In general, the algebra
$M(\Gamma,\mu)$ is a free product, with amalgamation over
a finite-dimensional abelian 
subalgebra corresponding to the vertex set, of algebras associated
to subgraphs `with one edge' (or actually a pair of dual edges). This
also yields `natural' examples of (i) a Fock-type model of an operator
with a free Poisson distribution; and (ii) $\C \oplus \C$-valued
circular and semi-circular operators.
\end{abstract}

\section{Preliminaries}

There has been a serendipitous convergence of  investigations being
carried out independently by us on the one hand, and by Guionnet,
Jones and Shlyakhtenko on the other - see [GJS1], [KS1], [KS2],
[GJS2]. As it has 
turned out, we have been providing independent proofs, from slightly
different viewpoints, of the same facts. Both the papers  [KS2] and
[GJS2], establish that a certain von Neumann algebra associated to a
graph is a free product with amalgamation of a family of von Neumann
algebras corresponding to simpler graphs. The amalgamated product
involved subgraphs indexed by vertices in [KS2], while the subgraphs
are indexed by edges in [GJS2]. 
This paper was motivated by trying to understand how the proof of our
result in [KS2] was also drastically simplfied by considering edges rather than
vertices. And, this third episode in our series seems to
have the following points in its favour: 
\begin{itemize}
\item It does make certain cumulant computations and consequent free
  independence assertions much more transparent. 
\item It brings to light a quite simple `Fock-type model' of free
  Poisson variables. 
\item By allowing non-bipartite graphs, we get the aesthetically
  pleasing fact mentioned in the abstract regarding the `flower on $n$
  petals'.
\end{itemize}

\bigskip 
We investigate, in a little more detail, the construction in [KS2]
which associated a von Neumann probability space to a weighted graph.
We begin by recalling the set-up:

\bigskip
By a weighted graph we mean a tuple $\Gamma = (V,E,\mu)$, where:
\begin{itemize}
\item $V$ is a (finite) set of vertices;
\item $E$ is a (finite) set of edges, equipped with `source' and `range' maps
$s,r:E \rar V$ and `(orientation) reversal' invoution map $E \ni e \mapsto
\tilde{e} \in E$ with $(s(e),r(e)) = (r(\tilde{e}),s(\tilde{e}))$; and
\item $\mu:V \rar (0,\infty)$ is a `weight or spin function' so
  normalised that $\sum_{u \in V} \mu^2(v) = 1$
\ei

We let $\CP_n = \CP_n(\Gamma)$ denote the set of paths of length $n$
in $\Gamma$ and  let $P_n(\Gamma)$ denote the vector space with basis $\{
[\xi] : \xi \in \CP_n(\Gamma)\}$. We think of 
$\xi = \xi_1 \xi_2 \cdots \xi_n$ as the `concatenation product' where
$\xi_i$ denotes the $i$-th edge of $\xi$. We write $F(\Gamma) =
\oplus_{n\geq 0}P_n(\Gamma)$ for the indicated direct sum, and equip
it with the following slightly complicated multiplication: if $\xi \in
\CP_m(\Gamma), \eta \in \CP_n(\Gamma)$, then $[\xi] \# [\eta] =
\sum_{k=0}^{\min(m,n)} [\zeta_k]$, where $\zeta_k \in \CP_{m+n-2k}$ is
defined by
\[ \zeta_k = \left\{ \ba{ll}
    \frac{\mu(v^\xi_m)}{\mu(v^\xi_{m-k})} [\xi_1\xi_2 \cdots \xi_{m-k}
    \eta_{k+1} \eta_{k+2} \cdots \eta_n] & \mbox{if }\xi_{m-j+1} =
    \widetilde{\eta_j} \forall 1 \leq j \leq k\\0 &
    \mbox{otherwise}\ea \right. \]

Here, and elsewhere, we adopt the convention that if $\xi \in
\CP_n$, then $\xi = \xi_1 \xi_2 \cdots \xi_n$ denotes concatenation
product, with $\xi_i \in E$ and 
we  write $s(\xi_i) = v^\xi_{i-1}$ (so also $r(\xi_i) = s(\xi_{i+1}) =
v^\xi_i$).

In particular, notice that
$\CP_0(\Gamma) = \{v : v \in V\}$, and that if $v=s(\xi),w=r(\xi)$
for some $\xi \in \CP_n$, and if $u_1,u_2 \in V$, then $[u_1] [\xi]
[u_2] = \delta_{u_1,v} \delta_{u_2,w} [\xi]$; and less trivially, if
$\xi \in \CP_1$ and $\eta \in \CP_m, m \geq 1$, then
\[ [\xi] \# [\eta] = \left\{ \ba{ll} ~0 & \mbox{if } r(\xi) \neq
  s(\eta)\\~[\xi \eta_1 ... \eta_m] & \mbox{if } r(\xi) = s(\eta)
  \mbox{ but } 
\xi \neq \widetilde{\eta_1}\\ 
~[\xi \eta_1 ... \eta_m] + \frac{\mu(r(\xi))}{\mu(s(\xi))}
[\eta_2\cdots\eta_m] & \mbox{if } \xi = \widetilde{\eta_1} 
\ea \right. \]

We define $\phi: F(\Gamma) \rar P_0$ by requiring that if $\xi \in
P_{n}$, then
\[\phi([\xi]) = \left\{ \ba{ll} ~0 & \mbox{if }  n > 0\\
~[\xi] & \mbox{if } n = 0 \ea
\right. \]
and finally define
\[\tau = \mu^2 \circ \phi\]
where we simply write $\mu^2$ for the linear extension to $P_0(\Gamma)$
which agrees with $\mu^2$ on the basis $\CP_0(\Gamma)$.

It was shown in [KS]\footnote{Actually, [KS] treated only the case of
  bipartite graphs, and sometimes restricted attention to the case of
  the Perron-Frobenius weighting; but for the the proof of statements made in this
  paragraph, none of those restrictions is necessary.}  that
$(F(\Gamma), \tau)$ is a tracial non-commutative 
*-probability space, with $e^* = \tilde{e}$, that the mapping $y
\mapsto xy$ extends to a $*$-algebra representation $F(\Gamma) \rar
\CL(L^2(F(\Gamma)),\tau)$ and that $M(\Gamma,\mu) = \lambda(F
\Gamma))^{\prime\prime} \subset \CL(L^2(F(\Gamma)),\tau)$ is in standard form.
Before proceeding further, it is worth noting that for $\xi, \eta
\in \cup_n \CP_n(\Gamma)$, we have
\[ \tau([\xi]\#[\eta]^*) = \delta_{\xi,\eta} \mu(r(\xi))\mu(s(\xi)) ~,\]
and hence, if we write $\{\xi\} = (\mu(s(\xi))\mu(r(\xi)))^{-\half}[\xi]$,
then $\{\{\xi\}: \xi \in \cup_{n \geq 0} \CP_n(\Gamma)\}$  is an orthonormal basis for
$\CH(\Gamma) = L^2(F(\Gamma), \tau)$.

\section{The building blocks}
Our interest here is the examination of just how $M(\Gamma,\mu)$ depends
on $(\Gamma, \mu)$. We begin by spelling out some simple examples,
which will turn out to be building blocks for the general case.

\begin{Example}\label{mgamegs} 
\ben
\item Suppose $|V|=|E|=1$, say $V=\{v\}$ and $E=\{e\}$. Then
  we must have $e=\tilde{e}, s(e)=r(e)=v,\mu(v)=1,\CP_n  = \{e^n\}$
  and $\{\xi(n)= \{e^n\}:n\geq 0\}$ (where $\{e^0\} = \{v\}$) is an 
  orthonormal basis for $\CH(\Gamma)$; and the definitions show that
   $x=\lambda(e)$ satisfies $x\xi_n=\xi(n+1) + \xi(n-1)$. Thus $x$ is
  a semi-circular element and $M(\Gamma) = \{x\}^{\prime\prime} \cong
  L\Z$.
\item Suppose $|V|=1, |E|=2$, say $V=\{v\}$ and $E=\{e_1,e_2\}$
  suppose $e_2=\widetilde{e_1}$. Then we must have
  $s(e_j)=r(e_j)=v,\mu(v)=1$. Further $\{\{e_1\},\{e_2\}\}$ is an
  orthonormal basis for $\CH_2 = P_1(\Gamma)$, and $P_n(\Gamma)$ is
  isomorphic to $\otimes^n \CH_2$. Thus $\CH(\Gamma)$ may be
  identified with the full Fock space $\CF(\CH_2)$ and the definitions show that
  $x_1=\lambda(e_1)$ may be identifed as $x_1 = l_1 + l_2^*$, where
  the $l_j$ denote the standard creation operators. It follows that
  $x_1$ is a circular element and $M(\Gamma) = \{x_1\}^{\prime\prime} \cong
  LF_2$.
\item Suppose $|V|=2, |E|=2$, say $V=\{v,w\}$ and
  $E=\{e,\widetilde{e}\}$ and suppose $s(e)=v, r(e)=w$ and
  $\mu(w) \leq \mu(v)$. Write $\rho = \frac{\mu(v)}{\mu(w)} (\geq
  1)$. If we let $p_v = \lambda([v]), p_w = \lambda([w]),$ it follows
  that $\CH_v = ran ~p_v$ (resp., $\CH_w = ran ~p_w$) has an
  orthonormal basis given by  
  $\{\{\eta(n)\}: n \geq 0\}$ (resp., $\{\{\xi(n)\}: n \geq 0\}$  where
  $\eta(n) \in \CP_n$ (resp., $\xi(n) \in \CP_n$) and $\eta(n)_k =
  e$ or $\tilde{e}$ (resp., $\xi(n)_k = \tilde{e}$ or ${e}$
  according as $k$ is odd or even). 

Writing $x = \lambda(e)$, we see that with respect to the
  decomposition $\CH(\Gamma) = \CH_v \oplus \CH_w$, the operator $x$
  has a matrix decomposition of the form
\[ x = \left[ \ba{ll} 0 & t\\0 & 0 \ea \right] \]
where $t \in \CL(\CH_w,\CH_v)$ is seen to be given by
\beast
t[\xi(n)] &=& x [\xi(n)]\\
&=& [e] \# [\tilde{e} e \tilde{e} e \cdots (n ~terms)] \\
&=& [\eta(n+1)] + \rho^{-1}[\eta(n-1)] ~;
\eeast
and hence,
\beast
t\{\xi(n)\} &=& (\mu(s(\xi(n))\mu(r(\xi(n)))^{-\half} t[\xi(n)]\\
&=& (\mu(w)\mu(r(\xi(n)))^{-\half} \left( [\eta(n+1)] + \rho^{-1}[\eta(n-1)]
\right)\\
&=& (\rho^{-1} \mu(v) \mu(r(\eta(n\pm 1)))^{-\half} \left( [\eta(n+1)] + \rho^{-1}[\eta(n-1)]
\right)\\
&=&  \rho^\half\{\eta(n+1)\} + \rho^{-\half}\{\eta(n-1)\}
\eeast

It is a fact - see Proposition \ref{abscont} - that $t^*t$ has
has absolutely continuous spectrum. This fact has two consequences:

(i) if $t = u|t|$ is the polar decomposition of $t$, then $u$ maps
$\CH_w$ isometrically onto the subspace $\CM = \overline{ran ~t}$ of
$\CH_v$, and if $z$ is the projection onto $\CH_v \ominus \CM$ then
$\tau(z) = \mu^2(v)-\mu^2(w)$; and

(ii) $W^*(|t|) \cong L\Z$.

Since $p_v + p_w = 1$ and $z \leq p_v$, the definitions are seen to
show that $M(\Gamma,\mu)$ is isomorphic to $\C \oplus M_2(L\Z)$ via
the unique isomorphism which maps $p_v,p_w,z,u$ and $|t|$, respectively, to
$(1,\left( \ba{cc} 1 & 
  0\\0 & 0\ea \right)), (0,\left( \ba{cc} 0 & 0\\0 & 1\ea \right)), 
  (1,\left( \ba{cc} 0 & 0\\0 & 0\ea \right)), (0,\left( \ba{cc} 0 &
  1\\0 & 0\ea \right)), $ and $(1,\left( \ba{cc} 0 & 0\\0 & a\ea
\right))$ for some positive $a$ with absolutely continuous spectrum which
generates $L\Z$ as a von Neumann algebra. (This must be compared with
Lemma 17 of [GJS2], bearing in mind that their $\mu$ is our $\mu^2$.)
\een
\end{Example}

\begin{Proposition}\label{abscont}
Let $\ell^2(\N)$ have its standard orthonormal basis $\{\delta_n : n
\in \N\}$. (For us, $\N = \{0,1,2, \cdots \}$.) Let $\ell \delta_n =
\delta_{n+1}$ denote the creation operator (or unilateral shift), with
$\ell^*\delta_n = \delta_{n-1}$ (where $\delta_{-1} = 0$). Let $\rho >
1$ and $t = \rho^\half \ell + \rho^{-\half}\ell^*$. Then, 
\ben
\item $t^*t$   leaves the subspace $\ell^2(2\N)$ invariant; 
\item $\delta_0$ is a cyclic vector for the restriction to $\ell^2(2\N)$  of $t^*t$, call it $a_\rho$;  and 
\item the (scalar) spectral measure of $a_\rho$ associated to
  $\delta_0$ is absolutely continuous with respect to Lebesgue
  measure; in fact $a_\rho$ has a free Poisson distribution.
\een
\end{Proposition}

\begin{proof}
A little algebra shows that
\beast
t^*t &=& (\rho^\half \ell^* + \rho^{-\half}\ell )(\rho^\half \ell + \rho^{-\half}\ell^*)\\
&=& \ell^2 + \ell^{*2} + (\rho + \rho^{-1}) - \rho^{-1} p_0 ~,
\eeast
where $p_0$ is the rank one projection onto $\C \delta_0$. It is seen that this operator leaves both subspaces $\ell^2(2\N)$ and $\ell^2(2\N+1)$ invariant, with its restrictions to these subspaces being unitarily equivalent to $\ell + \ell^* +  (\rho + \rho^{-1})  - \rho^{-1} p_0$ and $ \ell + \ell^* $ respectively. Since the spectral type does not change under scalar translation, we may assume without loss of generality that $a_\rho = \ell + \ell^* - \rho^{-1} p_0$ and establish that $a_0$ has absolutely continuous scalar spectral measure corresponding to $\delta_0$.

Write $a_0 = \ell + \ell^*$ so that $a_\rho = a_0 - \rho^{-1} p_0$. Let the scalar spectral measures of $a_0$ and $a_\rho$ be denoted by $\mu$ and $\mu_\rho$ respectively, and consider their Cauchy transforms given by
\[F_\lambda(z) = \langle (a_\lambda - z)^{-1}\delta_0, \delta_0 \rangle = \int_\R \frac{d\mu_\lambda (x)}{x - z}\]
for $\lambda \in \{0, \rho\}$ and
$z \in \C^+ = \{\zeta \in \C : Im (\zeta) > 0 \}$.

It follows from the resolvent equation that
\beast
F_\rho(z) &=&  \langle (a_\rho - z)^{-1}\delta_0, \delta_0 \rangle\\
&=&  \langle (a_0 - z)^{-1}\delta_0, \delta_0 \rangle +  \langle (a_\rho - z)^{-1}\rho^{-1}p_0(a_\lambda - z)^{-1}\delta_0, \delta_0 \rangle\\
&=& F_0(z) + \rho^{-1}F_\rho(z)F_0(z)~;
\eeast
Hence 
\be \label{frho} F_{\rho}(z) = \frac {F_0(z)} {1-\rho^{-1}F_0(z)} = \frac {\rho F_0(z)} {\rho - F_0(z)} \ee

It is seen from Lemma 2.21 of [NS] - after noting that the $G$ of that Lemma is the negative of the $F_0$ here - that $F_0(z) = \frac{z - \sqrt{z^2 - 4}}{2}$ where  $\sqrt{z^2 - 4}$ is a branch of that square root such that
$\sqrt{z^2 - 4} = \sqrt{z+2}\sqrt{z-2}$ where the two individual factors are respectively defined by using the branch-cuts $\{\mp 2 - it  : t \in (0, \infty)$. (This choice ensures that $lim_{|z| \rar \infty} F_0(z) = 0$, which is clearly necessary.) It follows that $F_0$, which is holomorphic in $\C^+$, actually extends to a continuous function on $\C^+ \cup \R$, and that if we write $f_0(a) = \lim_{b \downarrow 0} F_0(a + ib)$, then we have
\be \label{f0} 2 f_0(t) = \left\{ \ba{ll} - t + \sqrt{t^2 - 4} & \mbox{if } t \geq 2\\
- t + i \sqrt{4 - t^2} & \mbox{if } t \in [-2, 2]\\
- t - \sqrt{t^2 - 4} & \mbox{if } t \leq - 2 \ea \right. \ee
It is easy to check that $f_0$ is strictly increasing in $(-\infty,-2)$, as well as in in $(2, \infty)$, has non-zero imaginary part in $(-2,2)$, and satisfies $f(\R \setminus (-2,2)) = [-1,0) \cup (0,1]$. Since $\rho > 1$, we may deduce that $F_0(z) \neq \rho ~\forall z \in \C^+ \cup \R$, and hence that  also $F_\rho$ extends to a continuous function on $\C^+ \cup \R$ with equation (\ref{frho}) continuing to hold for all $z \in \C^+ \cup \R$. Writing $f_\lambda(t) = F_\lambda(t + i0)$ for $\lambda \in \{0, \rho\}$, we find that
\[f_\rho(t) = \frac{\rho f_0(t)}{\rho - f_0(t)} = \frac{1}{f_0(t)^{-1} - \rho^{-1} } ~,\]
and hence that
\beast
Im(f_\rho(t)) &=& - \frac{Im(f_0(t)^{-1})}{|f_0(t)^{-1} - \rho^{-1}|^2}\\
&=&  \frac{Im(f_0(t))}{|1 - f_0(t) \rho^{-1}|^2}\\
&=& \rho^2 \frac{Im(f_0(t))}{|f_0(t)  - \rho|^2}\\
&=& 1_{[-2,2]}(t) \frac{ \rho^2 \sqrt{4-t^2}}{2|f_0(t)  - \rho|^2} ~.
\eeast
Now, for $t \in [-2,2]$, we see that
\beast
|f_0(t)  - \rho|^2 &=& |\frac{-t + i \sqrt{4-t^2}}{2} - \rho|^2\\
&=& \frac{1}{4} \left( (t + 2 \rho)^2 + 4 - t^2 \right)\\
&=& \rho^2 + \rho t + 1~.
\eeast

It follows from Stieltje's inversion formula that our $a_\rho$ has absolutely continuous scalar spectral measure $\mu_\rho$, with density given by
\beast 
g_\rho(t) &=& \frac{1}{\pi} Im f_\rho(t)\\
&=& 1_{[-2,2]}(t) \frac{ \rho^2 \sqrt{4-t^2}}{2\pi (\rho^2 + \rho t + 1)}~.
\eeast

Hence the operator $t^*t = a_\rho + (\rho + \rho^{-1}) 1$ has has absolutely continuous scalar spectral measure, with density given by
\beast
g(t) &=& g_\rho(t - (\rho + \rho^{-1}))\\
&=&  1_{[(\rho + \rho^{-1})-2,(\rho + \rho^{-1})+2]}(t) \frac{\rho^2 \sqrt{4 - (t - (\rho + \rho^{-1})}^2}{2\pi \rho^{-2}(\rho^2 + \rho(t - \rho -\rho^{-1})+1)}\\
&=&  1_{[(\rho + \rho^{-1})-2,(\rho + \rho^{-1})+2]}(t) \frac{\rho^2 \sqrt{4 - (t - (\rho + \rho^{-1})}^2}{2\pi \rho^{-1}t }
\eeast
If we write $\lambda = \rho^2$ and $\alpha = \rho^{-1}$, we see that 
$\alpha(1 + \lambda)$ and recognise the fact that not only does $t^*t$ 
have absolutely continuous spectrum, but - by comparing with 
equation (12.15) of [NS] -even that it 
actually has a free Poisson distribution, with rate $\rho^2$ and 
jump size $\rho^{-1}$. However, we actually discovered this fact about
$t^*t$ having a free Poisson distribution with the stated  
$\lambda$ and $\alpha$ by a cute cumulant computation which we present
 in the final
 section, both for giving a combinatorial rather than analytic proof of this 
Proposition, and because we came across that proof first.
\end{proof}

\section{Some free cumulants}
Before proceeding with the further study of a general $(\Gamma,\mu)$,
we will need an alternative description of $M(\Gamma,\tau)$. 

Let $Gr(\Gamma) = \oplus_{n \geq 0}P_n(\Gamma)$ be
equipped with a $*$-algebra structure wherein
$[\xi] \circ [\eta] = [\xi\eta]$) and
$[\xi]^* = [\tilde{\xi}] = [\widetilde{\xi_n}\cdots\widetilde{\xi_1]}$
for $\xi \in \CP_n, \eta \in \CP_m$. It turns out - see
[KS]\footnote{The remark made in an earlier footnote, concerning
  assumptions regarding bipartiteness of $\Gamma$, applies here as well.} - that
$Gr(\Gamma)$ and $F(\Gamma)$ are isomorphic as $*$-algebras. While the
multiplication is simpler in $Gr(\Gamma)$, the trace $\tau$ on $F
(\Gamma)$ turns out, when transported by the above isomorphism, to be given
by a slightly more complicated formula. (It is what has been called
{\em the Voiculescu trace} by Jones et al.) We shall write $tr$ for
this transported trace on $Gr(\Gamma)$, and $\phi$ for the
$tr$-preserving conditional expectation of $M(\Gamma,\mu) (=
\lambda(Gr(\Gamma))^{\prime \prime})$ onto $P_0(\Gamma)$. We shall use
the same letter $\phi$ to denote restrictions to subalgebras which
contain $P_0(\Gamma)$.

We wish to regard $(Gr(\Gamma), \phi)$ as an operator-valued
non-commutative probability space over $P_0(\Gamma)$, our first order
of business being the determination of the $P_0(\Gamma)$-valued mixed
cumulants in $Gr(\Gamma)$.

\begin{Proposition}\label{Bvalcum}
The $P_0(\Gamma)$-valued mixed cumulants in $Gr(\Gamma)$ are given
thus:

$\kappa_n([e_1],[e_2], \cdots ,[e_n]) = 0$ unless $n=2$ and $e_2 =
\widetilde{e_1}$; and if $e_2 = \widetilde{e_1}$ with $s(e_1) = v,
r(e_1) = w$, then $\kappa_2([e_1],[\widetilde{e_1}])=\frac{\mu(w)}{\mu(v)}
[v]$.
\end{Proposition}

\begin{proof}
The proof depends on the `moment-cumulant' relations which guarantee
that in order to prove this proposition, it will suffice to establish
the following, which is what we shall do:

(a) Define $\kappa_n: (Gr(\Gamma))^n \rar P_0(\Gamma)$ to be the unique
multilinear map which is defined when the arguments
are tuples of paths as asserted in the proposition; note that (i)  it
is `balanced' over $P_0(\Gamma)$ in  
the sense that $\kappa_n(x_1,\cdots, x_{i-1}b, x_i, \cdots , x_n) =
\kappa_n(x_1,\cdots, x_{i-1}, bx_i, \cdots , x_n)$ for all $x_j \in Gr
(\Gamma), b \in P_0(\Gamma)$ and $1 < i \leq n$, and (ii) is
$P_0(\Gamma)$-bilinear meaning $\kappa_n(bx_1, x_2, \cdots, x_{n-1},
x_nb^\prime) = b \kappa_n(x_1, x_2, \cdots, x_{n-1},
x_n) b^\prime$ for all $x_j \in Gr(\Gamma), b,b^\prime \in P_0(\Gamma)$;

(b) define the `multiplicative extensions' $\kappa_\pi : (Gr(\Gamma))^n
\rar P_0(\Gamma)$ for $\pi \in NC(n)$ by requiring, inductively, that
if $[k,l]$ is an interval constituting a class of $\pi$, and if we
write $\sigma$ for the element of $NC(n-l+k-1)$ given by the
restriction of $\pi$ to $\{1,\cdots,k-1,l+1,\cdots,n\}$, so that `$\pi = \sigma
\bigvee 1_{[k,l]}$' then 
\beast \kappa_\pi(x_1, \cdots, x_n) &=&
\kappa_\sigma(x_1,\cdots,x_{k-1}\kappa_{l-k+1}(x_k, \cdots, x_l),
x_{l+1}, \cdots, x_n)\\
&=& \kappa_\sigma(x_1,\cdots,x_{k-1}, \kappa_{l-k+1}(x_k, \cdots,
x_l)x_{l+1}, \cdots, x_n);\eeast

(c) and verify that for any $e_1,\cdots,e_n \in \CP_1(\Gamma)$,
\be \label{Etpt}
\phi([e_1]\cdots [e_n]) = \sum_{\pi \in
  NC(n)}\kappa_\pi([e_1],[e_2],\cdots,[e_n]).
\ee

For this verification, we first assert that 
if $e_1,e_2, \cdots ,e_n \in E$ and $\pi \in NC(n)$, the quantity
$\kappa_\pi([e_1],[e_2],\cdots,[e_n])$ (yielded by the unique `multiplicative
extension' of the $\kappa_n$'s as in (b) above) can be non-zero only
if

(i) $e_1e_2\cdots e_n$ is a meaningfully defined loop based at
$s(e_1)$, meaning 
$f(e_i)=s(e_{i+1})$ for $1 \leq i \leq n$, with $e_{n+1}$ being
interpreted as $e_1$; 

(ii) $\pi \in NC_2(n)$ is a pair partition of $n$ (and in particular
$n$ is even), such that $\{i,j\} \in \pi \Rar e_j = \widetilde{e_i};$

\bigskip
and if that is the case, then,

\be \label{kn}
\kappa_\pi([e_1],[e_2],\cdots,[e_n])= \left( \prod_{\stackrel{\{i,j\} \in
      \pi}{i<j}} \frac{\mu(r(e_i)}{\mu(r(e_j)} \right) [s(e_1)]~. \ee

We prove this assertion by induction on $n$. This is trivial for $n=1$
since $\kappa_1 \equiv 0$. By the inductive definition of the
multiplicative extension, it is clear that if
$\kappa_\pi([e_1],[e_2],\cdots,[e_n])$ is to be non-zero, $\pi$ must
contain an interval class of the form $\{k,k+1\}$ such that $e_{k+1} =
\widetilde{e_k}$; if $\sigma$ denotes $\pi|_{\{1,2, \cdots, k-1, k+2,
  \cdots n\}}$ we must have
\beast \kappa_\pi([e_1],\cdots,[e_n]) &=& \frac{\mu(r(e_k))}{\mu(r(e_{k+1}))}
\kappa_\sigma([e_1],\cdots, [e_{k-1}] [s(e_{k})], [e_{k+2}], \cdots, [e_n])\\
&=& \frac{\mu(r(e_k))}{\mu(r(e_{k+1}))}
\kappa_\sigma([e_1],\cdots, [e_{k-1}], [s(e_{k})][e_{k+2}], \cdots, [e_n])\\
&=& \frac{\mu(r(e_k))}{\mu(r(e_{k+1}))}
\kappa_\sigma([e_1],\cdots, [e_{k-1}] [r(e_{k+1})], [e_{k+2}], \cdots, [e_n])~;\\
\eeast
and for this to be non-zero, we must have  $r(e_{k-1})=s(e_{k}) =
r(e_{k+1}) =s(e_{k+2})$, in which case we would have
\[\kappa_\pi([e_1],\cdots,[e_n]) = \frac{\mu(r(e_k))}{\mu(r(e_{k+1}))}
\kappa_\sigma([e_1],\cdots, [e_{k-1}], [e_{k+2}], \cdots, [e_n])~,\]
and the requirement that $\kappa_\sigma([e_1],\cdots, [e_{k-1}],
[e_{k+2}], \cdots, [e_n])$ be non-zero, along with the induction
hypothesis, finally completes the proof of the assertion.

Now, in order to verify equation \ref{Etpt}, it suffices to check that for
any $v \in V$, we have
\be \label{tptk} tr([e_1][e_2]\cdots [e_n][v])= \sum_{\pi \in
  NC(n)}tr(\kappa_\pi([e_1],[e_2],\cdots,[e_n])[v]).\ee

First observe that both sides of equation \ref{tptk} vanish unless
$e_1\cdots e_n$ is a meaningfully defined path with both source and
range equal to $v$ (since $tr$ is a trace and $[v]$ is idempotent).
In view of our description above of the multiplicative extension
$\kappa_\pi$, we need, thus, to verify that for such a loop, we have
\[ tr([e_1\cdots e_n]) = \sum_{\pi \in NC_2(n)} \left( \prod_{\stackrel{\{i,j\} \in
      \pi}{i<j}} \delta_{e_j, \widetilde{e_i}} \frac{\mu(r(e_i)}{\mu(r(e_j)} \right) \mu^2(s(e_1)),\]
but that is indeed the case (see equation (3) and the
proof of Proposition 5 in [KS1] ).
\end{proof}

In order to derive the true import of Proposition \ref{Bvalcum}, we
should first introduce some notation: 

For each dual pair $e, \tilde{e}$ of edges - with, say, $s(e)=v,
r(e)=w$ - we shall write $\Gamma_e = (V_e, E_e, \mu_e)$ where
$V_e=V, \mu_e=\mu$ and $E_e = \{e, \tilde{e}\}$ (with source, range
and reversal in $E_e$ as in $E$). If $e = \tilde{e}$, the above
definitions are to be suitably interpreted. Now for `the true import
of Proposition \ref{Bvalcum}':

\begin{Corollary}\label{amfrcor}
With the foregoing notation, we have:
\[Gr(\Gamma,\mu) = \ast_{P_0(\Gamma)} \{Gr(\Gamma_e,\mu_e): \{e,
\tilde{e}\} \subset E\} \]
and hence, also
\[M(\Gamma,\mu) = \ast_{P_0(\Gamma)} \{M(\Gamma_e,\mu_e): \{e,
\tilde{e}\} \subset E\}~. \]
\end{Corollary}

\begin{proof}
Proposition 3.3.3 of [S1] shows that 
if $A \stackrel{\phi}{\rar} B$ is a `non-commutative
  probability space  over $B$', if $\{A_i : i \in I\}$ is a family 
of subalgebras of $A$ containing $B$, such that $\{A_i : i \in I\}$
generates $A$, and if $G_i$ is a set of generators of the algebra $A_i$,
  then $A$ is
  the free product with amalgamation over $B$ of $\{A_i : i \in I\}$
  if and only if the mixed
  $B$-valued cumulants $\kappa_n(x_1, \cdots , x_n)$ vanish whenever
  $x_1, \cdots , \cdots x_n \in \cup_i 
  G_i$,   unless all the $x_i$
  belong to the same $G_k$ for some $k$. The desired assertion then
  follows from Proposition \ref{Bvalcum}.
\end{proof}

The following assertion, advertised in the abstract, is an immediate
consequence of Corollary \ref{amfrcor} and Examples \ref{mgamegs} (1)
and (2).

\begin{Corollary}
If $\Gamma_n$ denotes the `flower with $n$ petals' (thus $|V| = 1, |E|
= n$), then $M(\Gamma) \cong L\F_n$, independent of the reversal map
on $E$.
\end{Corollary}

\begin{Remark}
We may deduce from Proposition \ref{Bvalcum} that the $x = \lambda(e)$ of Example \ref{mgamegs} (3) is a $P_0(\Gamma)$-valued
circular operator, in the sense of [Dyk] (see Definition 4.1), with
covariance $(\alpha,\beta)$ where $\alpha(b)
= \phi(x^*bx)$ and $\beta(b) = \phi(xbx^*)$ for all $b \in \CP_0$ are
the completely positive self-maps of $P_0(\Gamma) (= \C p_v \oplus \C
p_w)$ induced by the matrices 
\[\alpha = \left[ \ba{cc} 0 & \rho^{-1}\\ 0 & 0 \ea \right] \mbox{ and }
 \beta = \left[ \ba{cc} 0 & 0\\ \rho & 0 \ea \right]  ~.\]

If $s = x+x^*$, it follows then that $s$ is a $P_0(\Gamma)$-valued
semi-circular element (since $\kappa_n(sb_1,sb_2,\cdots sb_{n-1},s) = 0$
unless $n=2$ and $\kappa_2(sb,s) = \eta(b)$  where $\eta$ is the 
(completely) positive self-map of  $\C \oplus \C$  induced by the matrix
\[\eta = \left[ \ba{cc} 0 & \rho^{-1}\\ \rho & 0 \ea \right] ~.\]
\end{Remark}

\section{Narayana numbers}

Recall the Narayana numbers $N(n,k)$ defined for all  $n,k \in {\mathbb N}$ with
$1 \leq k \leq n$ by
$$
N(n,k) = |\{ \pi \in NC(n) : |\pi| = k \}|.
$$
Define the associated polynomials $N_n$ by
$$
N_n(T) = \sum_{k=1}^n N(n,k) T^k.
$$

Recall also that a random variable in a non-commutative probability space
$(A,\tau)$ is said to be free Poisson with rate $\lambda$ and jump size
$\alpha$ if its free cumulants are given by $\kappa_n = \lambda \alpha^n$
for all $n \in {\mathbb N}$. An easy application of the moment-cumulant
relations shows that an equivalent condition for a random variable to
be free Poisson with rate $\lambda$ and jump size
$\alpha$ is that its moments are given by $\mu_n = \alpha^n N_n(\lambda)$
for all $n \in {\mathbb N}$. 

We now illustrate an application of this characterisation of a free Poisson
variable in the situation of \S 2, Example 2.1 (3).
There, $x = \lambda(e)$ has a matrix decomposition involving 
$t \in {\cal L}({\cal H}_w,
{\cal H}_v)$ where $t^*t$ was shown to have a free Poisson distribution.
We will verify below by a cumulant computation that $t^*t$ is free Poisson with
rate $\rho^2$ and jump size $\rho^{-1}$ in the non-commutative probability
space $p_wM(\Gamma,\mu)p_w$.

Begin by observing that $x^*x$ has a non-zero entry only in the $w$-corner
and that this entry is $t^*t$. Thus the trace in
$M(\Gamma,\mu)$ of $(x^*x)^n$ 
is $\mu^2(w)$ times the trace - call it $tr_w$ - in
$p_wM(\Gamma,\mu)p_w$ of $(t^*t)^n$. We now
compute $tr((x^*x)^n) = tr(([e]^*[e])^n)$.

First apply the moment-cumulant relations and Proposition 3.1 to conclude
that
$$
\phi(([e]^*[e])^n) = \sum_{\pi \in NC(2n)} \kappa_\pi([e]^*,[e],\cdots,[e]^*,[e]).
$$
While this sum ranges over all $\pi  \in NC(2n)$, Proposition 3.1 enables
us to conclude that unless $\pi$ is a non-crossing pair partition, its
contribution vanishes.
Thus we have:
\[\phi(([e]^*[e])^n) = \sum_{\pi \in NC_2(2n)}
\kappa_\pi([e]^*,[e],\cdots,[e]^*,[e]).\] 

Now we use the well-known bijection between non-crossing pair partitions
(or equivalently, Temperley-Lieb diagrams) on $2n$ points and all
non-crossing partitions on $n$ points. 
We will denote this bijection
as $\pi \in NC_2(2n) \leftrightarrow \tilde{\pi} \in NC(n)$.
This is illustrated by example in Figure \ref{bij} for $\pi =
\{\{1,8\},\{2,5\},\{3,4\},\{6,7\},\{9,12\},\{10,11\}\}$ and may be 
summarised by saying that the black regions of the Temperley-Lieb diagram for $\pi \in NC_2(2n)$
correspond to the classes of $\tilde{\pi} \in NC(n)$.
\begin{figure}[htpb]
\begin{center}
\includegraphics[height=2cm]{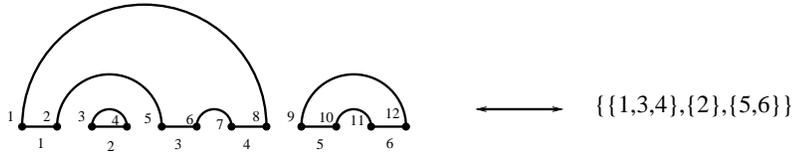}
\caption{$\pi \in NC_2(12) \leftrightarrow \tilde{\pi} \in NC(6)$}
\end{center}
\label{bij}
\end{figure}
Note that in Figure \ref{bij} the numbers above refer to the vertices while those below
refer to the black segments.

It follows from Proposition 3.1 that for
any $\pi \in NC_2(2n)$, the term $\kappa_\pi([e]^*,[e],\cdots,[e]^*,[e])$ is a 
scalar multiple of $p_w$ where the scalar is given by a
product
of $n$ terms each of which is $\rho = \frac{\mu(v)}{\mu(w)}$ or 
$\rho^{-1} = \frac{\mu(w)}{\mu(v)}$. Classes of $\pi$ for which the smaller
element is odd give $\rho$, while those for
which the smaller element is even give $\rho^{-1}$.
Thus $\kappa_\pi([e]^*,[e],\cdots,[e]^*,[e])$ evaluates to 
$\rho^{(|\pi|_{odd} - |\pi|_{even})} p_w = \rho^{(2|\pi|_{odd} - n)} p_w$, where, of course, $|\pi|_{odd}$ (resp. $|\pi|_{even}$) denotes
the number of classes of $\pi$ whose smaller element is odd (resp. even).

Our main combinatorial observation is contained in the following simple
lemma.

\begin{lemma} 
For any $\pi \in NC_2(2n)$, $|\pi|_{odd} = |\tilde{\pi}|$.
\end{lemma}

\begin{proof}
We induce on $n$ with the basis case $n=1$ having only one $\pi$
with $|\pi|_{odd} = |\tilde{\pi}| = 1.$ For larger $n$, consider
a class of $\pi$ of the form $\{i,i+1\}$, and remove it to get
$\rho \in NC_2(2n-2)$. A moment's thought shows that if $i$ is odd
then $|\pi|_{odd} = |\rho|_{odd} + 1 = |\tilde{\rho}| + 1 = |\tilde{\pi}|$,
while if $i$ is even then $|\pi|_{odd} = |\rho|_{odd} = |\tilde{\rho}| = |\tilde{\pi}|.$
\end{proof}

Thus:
\begin{eqnarray*}
\phi(([e]^*[e])^n) &=& \sum_{\pi \in NC_2(2n)} \rho^{(2|\pi|_{odd} -
  n)} p_w\\
&=& \sum_{\tilde{\pi} \in NC(n)} \rho^{(2|\tilde{\pi}| - n)} p_w\\
&=& \sum_{k=1}^{n} \sum_{\{\tilde{\pi} \in NC(n) : |\tilde{\pi}| = k\} } \rho^{2k-n} p_w\\
&=& \sum_{k=1}^{n} N(n,k)\rho^{2k-n} p_w
\end{eqnarray*}

Hence $tr(([e]^*[e])^n) = \sum_{k=1}^{n} N(n,k)\rho^{2k-n} \mu^2(w)$ and thus
$tr_w((t^*t)^n) = \sum_{k=1}^{n} N(n,k)\rho^{2k-n}$. Now the characterisation of
free Poisson elements in terms of their moments shows that $t^*t$ is free
Poisson with rate $\rho^2$ and jump size $\rho^{-1}$.

\begin{Remark}\label{frpois}
{\rm 
\ben
\item
Thus, for $t = \rho^\half \ell + \rho^{-\half} \ell^*$, we have shown
that $t^*t$ is a free Poisson element with rate $\rho^2$ and jump size
$\rho^{-1}$. By scaling by an appropriate constant, we can 
similarly obtain such simple Fock-type models of free Poisson elements
with arbitrary jump size and rate.
\item
Similar scaling, and the fact that $e^{i \theta} \ell$ is unitarily
  equivalent to $\ell$ (by a unitary operator which fixes $\delta_0$)
show that, in fact, if $t = a\ell + b\ell^*$ for any $a,b \in \C$, then
$t^*t$ is a free Poisson element.
\een }
\end{Remark}

{\bf Acknowledgement:} We would like to thank M. Krishna for patiently
leading us through the computation of Cauchy transforms of rank-one
perturbations as we struggled with an apparent contradiction, which
was finally resolved when we realised a problematic minus sign
stemming from a small mistake in choice of square roots. (We claim no
originality for this problem, for the same incorrect sign also
surfaces on page 33 of [NS] - cf. our formula (\ref{f0}) and the
formula there for $g$ , when $-\infty < t < -2$.)

\end{document}